\documentclass{amsart}
\usepackage{amsfonts}

\newtheorem{theorem}{Theorem}
\theoremstyle{plain}

\newtheorem{lemma}{Lemma}

\newtheorem{problem}{Problem}
\newtheorem{proposition}{Proposition}
\newtheorem{remark}{Remark}

\numberwithin{equation}{section}
\input{tcilatex}

\begin{document}
\title[L\'{e}vy processes]{L\'{e}vy processes, martingales, reversed
martingales and orthogonal polynomials}
\author{Pawe\l\ J. Szab\l owski}
\address{Department of Mathematics and Information Sciences,\\
Warsaw University of Technology\\
ul Koszykowa 75, 00-662 Warsaw, Poland}
\email{ pawel.szablowski@gmail.com}
\date{November, 2012}
\subjclass[2000]{Primary 60G51 60G44; Secondary11B68}
\keywords{L\'{e}vy processes, polynomial martingales, orthogonal
polynomials, reversed martingales, harnesses, tangent numbers.}

\begin{abstract}
We study class of L\'{e}vy processes having distributions being
indentifiable by moments. We define system of polynomial martingales \newline
$\left\{ M_{n}(X_{t},t),\mathcal{F}_{\leq t}\right\} _{n\geq 1},$ where $%
\mathcal{F}_{\leq t}$ is a suitable filtration defined below. We present
several properties of these martingales. Among others we show that $%
M_{1}(X_{t},t)/t$ is a reversed martingale as well as a harness. Main
results of the paper concern the question if martingale say $M_{i}$
multiplied by suitable determinstic function $\mu _{i}(t)$ is a reversed
martingale. We show that for $n\geq 3$ $M_{n}(X_{t},t)$ is a reversed
martingale (or orthogonal polynomial) only when the L\'{e}vy process in
question is Gaussian (i.e. is a Wiener process). We study also a more
general question if there are chances for a linear combination (with
coefficients depending on $t)$ of martingales $M_{i},$ $i\allowbreak
=\allowbreak 1,\ldots ,n$ to be reversed martingales. We analyze case $%
n\allowbreak =\allowbreak 2$ in detail listing all possible cases.
\end{abstract}

\maketitle

\section{Introduction}

Let us recall that L\'{e}vy processes $\left\{ X_{t}\right\} _{t\geq 0}$ are
such stochastic processes that start from zero i.e. $X_{0}\allowbreak
=\allowbreak 0$ a.s. and have stationary and independent increments which
means that distribution of $X_{t}-X_{s}$ is the same as that of $X_{t-s}$
for all $0\leq s\leq t$ and $X_{t}-X_{s}$ is independent on $X_{u}-X_{v}$
whenever $0\leq s<t\leq v<u.$

This paper deals with those L\'{e}vy processes that posses all moments, more
precisely we assume that the distributions of $X_{t},$ $t\geq 0$ are
identifiable by their moments. Among other advantages this assumption allows
to define a family of polynomial functions constructed of observations of
the process. We examine such properties of these polynomials as being a
martingale, a reversed martingale or a harness. The martingale theory is a
very developed method of analysis of stochastic processes hence indicating
martingales that can be constructed from the L\'{e}vy process we broaden the
spectrum of tools that are at hand in analysis of a given L\'{e}vy process.

One can define many families of polynomials for L\'{e}vy processes with
existing all moments. The most popular ones are the Kaillath--Segall
polynomials (see \cite{Segal76}, \cite{Lin81}, \cite{Sol08}, , \cite{Yabl08}%
) connected with a path' structure of the process and the properties of the
multiple integrals of the process. There are also so called Teugels
polynomials (see \cite{SchTeu98}, \cite{Sch2000}) associated with the
properties of the L\'{e}vy measure of the process.

As stated above we are seeking such polynomial functions $M_{n}(X_{t},t)$ of
the process's observations $X_{t}$ at $t$ that are martingales. We indicate
conditions under which these polynomials multiplied by some deterministic
functions of the time parameter or their linear combinations with depending
on $t$ coefficients are the reversed martingales or constitute a family of
orthogonal, polynomial martingales. We give some properties of the so called
'connection coefficients' between polynomial martingales and orthogonal
polynomials of the marginal distribution.

We also analyze the structure of the so called 'angular brackets' of the
martingales $M_{n}$ i.e. functions $p_{n}(t)\allowbreak =\allowbreak
EM_{n}(X_{t},t)^{2}.$

Of course there exist relations of our martingales with Kaillath--Segall
polynomials (see \cite{Segal76}) or Yablonski's polynomials (see \cite%
{Yabl08}). In 2011 during a seminar presentations in Innsbruck J.L. Sol\'{e}
constructed polynomial martingales using Bell's (or Yablonski's)
polynomials. This was based on two papers \cite{SolU08} and \cite{SolUz08t}.
We present many more properties of these martingales than it was mentioned
in Sol\'{e}'s and Utzet papers and presentation. They include expansion of
some products of these martingales in linear combinations of them. Those
useful technical results are presented in Lemma \ref{pom2}. We study also
relation of polynomial martingales $M_{n}$ to the system of orthogonal
polynomials of the marginal distributions. Some results in this topic are
presented in Proposition \ref{ort}. Of course on the way we point out
relationship with Yablonski's polynomials.

The paper is organized as follows. The next Section \ref{main} contains our
main results. It is divided into two subsections. Subsection \ref{gen}
contains properties of the family of polynomial martingales $\left\{
M_{n}\right\} $ while Subsection \ref{rev} our main results answering
questions if  polynomial martingales $\left\{ M_{n}\right\} $ are harnesses
or have reversed martingale property (respectively Theorems \ref{first} and %
\ref{gt3}). We consider also question when linear combinations of
martingales $\left\{ M_{n}\right\} $ are reversed martingales  (Theorem \ref%
{glowne}) as well as we study the relationship between polynomials
orthogonal with respect to the marginal distributions and polynomial
martingales $\left\{ M_{n}\right\} .$ Section \ref{open} contains some open
problems that can be solved using technic presented in the paper and which
we leave to more talented researchers. Finally Section \ref{dowody} contains
some technical, auxiliary results as well as longer, tedious proofs.

At last let us mention the fact that while analyzing consequences of the
assumption that $\mu (t)M_{2}(X_{t},t)$ is the reversed martingale we had to
prove, believed to be new, interesting property of the so called tangent
numbers (see (\ref{TnaT})), numbers closely related to Bernoulli numbers.

\section{Polynomial martingales\label{main}}

Let us formulate assumptions that will be in force throughout the paper.

On the probability space $(\Omega ,\mathcal{F},P)$ there is defined a L\'{e}%
vy stochastic process $\mathbf{X\allowbreak =\allowbreak (}X_{t})_{t\geq 0,}$
i.e. time homogeneous process with independent increments, continuous with
probability.

We define filtrations $\mathcal{F}_{\leq s}\allowbreak =\allowbreak \sigma
(X_{u}:u\leq s)$ for $s>0,$ $\mathcal{F}_{\geq s}\allowbreak =\allowbreak
\sigma (X_{u}:u\geq s)$ and $\mathcal{F}_{s,u}\allowbreak =\allowbreak
\sigma (X_{t}:t\notin (s,u)).$

We want to stress that all equalities between random variables are
understood to be with probability $1.$ Hence we drop abbreviation a.s.
usually following equality between random variables for the clarity of
exposition.

We will be interested only in those L\'{e}vy processes which posses all
moments. Such processes constitute a subclass of the class of all L\'{e}vy
processes and the main tool of analyzing them are the moment functions.
Hence we will not refer to the L\'{e}vy measure which is traditionally used
in the analysis of L\'{e}vy processes. Instead we will use Kolmogorov's
characterization of the infinitely divisible distributions as presented e.g.
in \cite{GnKol49} to study our class of L\'{e}vy processes. Of course the
two approaches are closely related since one can get all moment functions of
the process knowing its characteristic function. We will use moment
functions since they constitute a very natural tool of examining the
analyzed class of processes, for the sake of completeness of the paper and
also in order to illustrate the usage of the recently obtained results of
the paper \cite{SzablPoly}.

Let us denote by $m_{n}(t)$ the $n-$th moment of the process i.e. $%
m_{n}(t)\allowbreak =\allowbreak EX_{t}^{n}.$ We will assume that for all $%
n\geq 0$ functions $m_{n}(t)$ exist and are well defined.

Let us recall that a sequence $\left\{ \alpha _{n}\right\} _{n\geq 0}$ of
real numbers is called \emph{a moment sequence} iff every $(n+1)\times (n+1)-
$ matrix defined by $\left[ \alpha _{i+j}\right] _{0\leq i,j\leq n}$ is
positive definite. It is known that then exists a positive measure $d\beta $
such that $\alpha _{n}\allowbreak =\allowbreak \int x^{n}d\beta (x).$ Let us
also recall that not every moment sequence defines uniquely the measure
whose moments the elements of this sequence constitute. In order that this
measure be uniquely defined certain restrictions on he moment sequence have
to be imposed. The most popular one is the Carleman's condition stating that
if 
\begin{equation}
\sum_{n\geq 0}\frac{1}{\alpha _{2n}^{1/2n}}=\infty ,  \label{carl}
\end{equation}%
then the moment sequence $\left\{ \alpha _{n}\right\} _{n\geq 0}$ defines
its measure uniquely. Another criterion is that 
\begin{equation*}
\int \exp \left( y\left\vert x\right\vert \right) d\beta \left( x\right)
<\infty ,
\end{equation*}%
for some $y>0.$

In the sequel we will assume that $\forall t\geq 0$ sequence $\left\{
EX_{t}^{n}\right\} _{n\geq 0}$ defines marginal measure uniquely. For the
compact introduction see e.g. first sections of \cite{Sim98}. The discussion
of how assumptions we are making in order to assure the existence of
characteristic function of moments and the above mentioned assumptions
assuring identifiability of distribution by its moments is done in Remark %
\ref{kolmog}.

\subsection{General properties\label{gen}}

We have the following set of easy observations some of which are known. We
present them here for the completeness of the paper.

\begin{proposition}
\label{momenty}i) 
\begin{equation}
m_{n}(s+t)\allowbreak =\allowbreak \sum_{j=0}^{n}\binom{n}{j}%
m_{j}(s)m_{n-j}(t),  \label{mom}
\end{equation}%
for all $n\geq 0$ and $s,t\geq 0.$

ii) Let $Q(t;x)\allowbreak =\allowbreak \sum_{j\geq
0}m_{j}(t)x^{j}/j!\allowbreak $ be the characteristic function of the moment
functions, then 
\begin{equation*}
Q(t;x)\allowbreak =\allowbreak \exp (tf(x)),
\end{equation*}%
with $f(x)\allowbreak =\allowbreak \sum_{k\geq 1}c_{k}x^{k}/k!$.
Coefficients $c_{i},$ $i\allowbreak =\allowbreak 1,\ldots $ are such that
for every $t\geq 0$ the sequence $\left\{ m_{n}(t)\right\} _{n\geq 0}$ is
the moment sequence.

iii) $c_{1}t\allowbreak =\allowbreak EX_{t},$ $\limfunc{var}%
(X_{t})\allowbreak =\allowbreak c_{2}t.$ Let $\hat{m}_{n}(t)\allowbreak
=\allowbreak E(X_{t}-c_{1}t)^{n}$ be the central moment sequence. Then $%
\sum_{j\geq 0}\hat{m}_{n}(t)x^{j}/j!\allowbreak =\allowbreak \exp
(t(f(x)-c_{1}x)).$

iv) Moment functions $m_{n}(t)$ satisfy the following set of differential
equations: $\forall n>0,$ $t>0$ 
\begin{equation}
m_{n}^{^{\prime }}(t)=\sum_{j=1}^{n}\binom{n}{j}c_{j}m_{n-j}(t).
\label{m_prim}
\end{equation}

v) $\sum_{j=0}^{n}\binom{n}{j}m_{n-j}(-s)m_{j+i}(s)\allowbreak =\allowbreak
\left. \frac{\partial ^{n}}{\partial u^{n}}(\exp (-sf(u))\frac{\partial ^{i}%
}{\partial u^{i}}\exp (sf(u)))\right\vert _{u=0}$
\end{proposition}

\begin{proof}
Is shifted to Section \ref{dowody}.
\end{proof}

\begin{remark}
Looking at assertion ii) and confronting it with the well known definition
of so called cumulants i.e. coefficients of the power series expansion of
the function $\log \int \exp (xy)d\beta (y)$ we see that coefficients $c_{i}$
are cumulants of the distribution of $X_{1}.$
\end{remark}

Let us remark that (\ref{mom}) is well known. We have recalled it here for
the sake of completeness.

\begin{remark}
\label{-t}Since $c_{0}\allowbreak =\allowbreak 0$ from the expansion $\exp
(tf\left( x)\right) \allowbreak =\allowbreak \sum_{n\geq
0}(tx)^{n}(f(x)/x)^{n}/n!$ we deduce that coefficient by $x^{n}$ is a
polynomial in $t$ of degree at most $n.$ Thus we can define moment functions
for non-positive $t$. Consequently (\ref{mom}) is true for all $t,s\in 
\mathbb{R}$.
\end{remark}

Hence as a corollary we have the following observation.

\begin{proposition}
\label{mart}i) Let us define for all $n$ and $t>0:$ 
\begin{equation}
M_{n}\left( x,t\right) \allowbreak =\allowbreak \sum_{j=0}^{n}\binom{n}{j}%
m_{n-j}(-t)x^{j}.  \label{martingales}
\end{equation}%
Then for all $n$ and $t>s>0:$ 
\begin{equation}
E(M_{n}(X_{t},t)|\mathcal{F}_{\leq s})\allowbreak =\allowbreak
M_{n}(X_{s},s).  \label{mart1}
\end{equation}

ii) Characteristic function of polynomials $\left\{ M_{n}\left( x,t\right)
\right\} $ is the following:%
\begin{equation*}
\sum_{n\geq 0}\frac{r^{n}}{n!}M_{n}(x,t)=\exp (rx-tf(r))\overset{df}{=}%
\mathcal{N}_{t}(x,r).
\end{equation*}

iii) We have for $t>s>0:$%
\begin{equation*}
E(\mathcal{N}_{t}(X_{t},r)|\mathcal{F}_{\leq s})=\mathcal{N}_{s}(X_{s},r),
\end{equation*}%
hence $(\mathcal{N}_{t}(X_{t},r),\mathcal{F}_{\leq t})$ is a martingale
known as 'exponential martingale'. Besides $E\mathcal{N}_{t}(X_{t},r)%
\allowbreak =\allowbreak 1.$
\end{proposition}

\begin{proof}
We have $E(M_{n}(X_{t},t)|\mathcal{F}_{\leq s})\allowbreak \allowbreak
=\allowbreak \sum_{j=0}^{n}\binom{n}{j}m_{n-j}(-t)E(X_{t}-X_{s}+X_{s})^{j}%
\allowbreak =\allowbreak \sum_{j=0}^{n}\binom{n}{j}m_{n-j}(-t)\sum_{k=0}^{j}%
\binom{j}{k}X_{s}^{k}m_{k-j}(t-s)\allowbreak =\allowbreak \sum_{k=0}^{n}%
\binom{n}{k}X_{s}^{k}\sum_{j=k}^{n}\binom{n-k}{j-k}m_{k-j}(t-s)m_{n-j}(-t)%
\allowbreak =\allowbreak \sum_{k=0}^{n}\binom{n}{k}X_{s}^{k}m_{n-k}(-s)$ by (%
\ref{mom}).

ii) It easily follows from Proposition \ref{momenty}, ii) and (\ref%
{martingales}). iii) Follows directly from (\ref{mart1}).
\end{proof}

\begin{remark}
\label{other}Notice that polynomial martingales $\left\{ M_{n}\right\}
_{n\geq 0}$ are not the only polynomial martingales of the given L\'{e}vy
process. In fact family of polynomials defined by 
\begin{equation*}
\tilde{M}_{n}(X_{t},t)=\sum_{j=1}^{n}b_{n,j}M_{j}\left( X_{t},t\right) ,
\end{equation*}%
where coefficients $\left\{ b_{n,j}\right\} $ do not depend on $t,$
constitute another family of polynomial martingales.
\end{remark}

\begin{remark}
Assertion ii) of the above mentioned Proposition appeared earlier in \cite%
{SolU08}.
\end{remark}

\begin{remark}
\label{Kolmog}Coefficients $c_{i}$ can be identified with the moments of
Kolmogorov's measure $dK$ of the analyzed L\'{e}vy process. Recall that
since we deal with the process that has finite variance we can use the L\'{e}%
vy canonical form of the infinitely divisible distribution in the equivalent
(Kolmogorov's) form (see e.g. \cite{GnKol49}, p.93, (10)). Applying
appropriate formula for $t\allowbreak =\allowbreak -ix$ we get 
\begin{equation}
E\exp (xX_{t})\allowbreak =\allowbreak \exp (tf(x))\allowbreak =\allowbreak
\exp (tc_{1}x+t\int_{\mathbb{-\infty }}^{\infty }\frac{(\exp (xy)-1-xy)}{%
y^{2}}dK(y),  \label{kolmog}
\end{equation}%
where $K(y)$ is a non-decreasing function with bounded variation such that $%
K(-\infty )\allowbreak =\allowbreak 0$ and $K(\infty )\allowbreak
=\allowbreak \int_{\mathbb{R}}dK(y)\allowbreak =\allowbreak \limfunc{var}%
(X_{1})\allowbreak =\allowbreak c_{2}.$
\end{remark}

Following this remark we have the following Proposition exposing
relationship between coefficients $\left\{ c_{j}\right\} _{j\geq 1}$ and
moments of the measure $dK.$ 

\begin{proposition}
\label{dK}i) 
\begin{equation*}
c_{i}\allowbreak =\allowbreak \int_{\mathbb{R}}y^{i-2}dK(y),
\end{equation*}%
consequently  $c_{i}/c_{2}$ is the $i-2$ moment of the probability measure $%
\frac{1}{c_{2}}dK(y)$, for $i\geq 3.$

ii) $(c_{4}/c_{2})^{1/2}\allowbreak \leq \allowbreak
(c_{6}/c_{2})^{1/4}\allowbreak \leq \allowbreak \ldots \leq \allowbreak
(c_{2k+2}/c_{2})^{1/2k}\allowbreak \leq \ldots $

iii) $c_{4}\allowbreak -\allowbreak c_{3}^{2}/c_{2}\allowbreak \allowbreak
\geq \allowbreak 0$ since $c_{4}/c_{2}\allowbreak -\allowbreak
(c_{3}/c_{2})^{2}$ is the variance of the measure $\frac{1}{c_{2}}dK.$

iv) If $c_{2k}\allowbreak =\allowbreak 0$ for some $k\geq 2$ then $dK$ must
be degenerated and concentrated at $0$ (consequently $c_{i}\allowbreak
=\allowbreak 0$ for $i\geq 3$) so we deal with the Gaussian case since 
\newline
\begin{equation*}
\exp (c_{1}xt\allowbreak +t\allowbreak c_{2}x^{2}/2)\allowbreak =\allowbreak
\int \exp (xy)\frac{1}{\sqrt{2\pi c_{2}t}}\exp (-\frac{(y-xc_{1}t)^{2}}{%
2c_{2}t})dy.
\end{equation*}

v) If $c_{4}/c_{2}\allowbreak -\allowbreak (c_{3}/c_{2})^{2}\allowbreak
=\allowbreak 0$ then $dK$ is degenerated and concentrated at $\frac{c_{3}}{%
c_{2}},$ consequently $c_{i}\allowbreak =\allowbreak c_{3}^{i-2}/c_{2}^{i-3}$
for $i\geq 3.$  We deal in this case with a mixture of the modified Poisson
(i.e. concentrated at points $nc_{3}/c_{2},$ $n\geq 3$) and Gaussian
distributions. The mixture depends on the relationship between $c_{2}$ and $%
c_{3}.$
\end{proposition}

\begin{proof}
i) We confront (\ref{kolmog}) with the definition of the coefficients $c_{i}.
$ii) We use Jensen's inequality. iii), iv), Are trivial. v) We confront
assertion iii) with (\ref{kolmog}). 
\end{proof}

\begin{remark}
\label{identif}Let us recall result from \cite{Sato99} stating that measure $%
K$ defines one dimensional marginal measures uniquely. Hence if Kolmogorov
measure $K$ is unidentifiable by moments then the same must be true with
marginal measures and conversely. Notice also that if measure $K$ is
identifiable by moments then expression $\int_{\mathbb{-\infty }}^{\infty }%
\frac{(\exp (xy)-1-xy)}{y^{2}}dK(y)$ is finite, consequently $\log (E\exp
(xX_{t}))$ is finite in some neighborhood of zero and we deal with the so
called 'small exponential moments' case, the situation often considered by
researchers working on L\'{e}vy processes.
\end{remark}

Following the above mentioned Remarks and interpretation of the coefficients 
$c_{i}$ we will assume from now on that these coefficients are such that the
Kolmogorov's measure $dK$ is determined by them completely.

Since coefficients $c_{i},$ $i\geq 1$ determine L\'{e}vy process with finite
all moments completely we will use notation $\mathbf{X(}\left\{
c_{i}\right\} ),$ $\mathbf{X(c),}$ or finally $\mathbf{X(\{}%
c_{1},c_{2},\ldots \})$ to denote L\'{e}vy process with parameters $%
\{c_{1},c_{2},\ldots \}.$

\begin{remark}
Taking into account interpretation and properties of the coefficients $c_{i}$
given above we can refer to the martingale characterizations given by Weso\l %
owski in \cite{Wes90P}. One of them is obviously wrong. Namely the
characterization of the Poisson process by the form of first three
polynomial martingales is not true. This is so since from the martingale
conditions considered by Weso\l owski in Theorem 1. of \cite{Wes90P} it
follows that $c_{1}\allowbreak =\allowbreak c_{2}\allowbreak =\allowbreak
c_{3}\allowbreak =\allowbreak 1.$ As the above Remark shows it is not enough
to impose that all $c_{i}\allowbreak =\allowbreak 1$ for $i\geq 4$ which
would lead to the Poisson process with parameter $1$ as indicated in Remark %
\ref{Kolmog}, iv).

On the other hand the second martingale characterization of the Wiener
process (within the class of L\'{e}vy processes) by the first four
polynomial martingales given by Theorem 3. of \cite{Wes90P} is true since
the form of these martingales impose that $c_{3}\allowbreak =\allowbreak
c_{4}\allowbreak =\allowbreak 0.$ As it can be seen from Remark \ref{Kolmog}%
, iii) it is enough to deduce that then all $c_{i}\allowbreak =\allowbreak 0$
for $i\geq 4.$
\end{remark}

\begin{remark}
In \cite{Yabl08}(2.1) Yablonski defined family of polynomials $%
P_{n}(x_{1},\ldots ,x_{n})$ of the increasing numbers of variables by the
expansion%
\begin{equation}
\exp (\sum_{k\geq 1}\frac{(-1)^{k-1}x_{k}}{k}z^{k})=\sum_{n\geq
0}z^{n}P_{n}(x_{1},\ldots ,x_{n}).  \label{Yab}
\end{equation}%
He proved validity of the above expansion for $\left\vert z\right\vert
<1/\lim \sup_{k\longrightarrow \infty }\left\vert x_{k}\right\vert ^{1/k}$
and also gave some properties of these polynomials. Comparing (\ref{Yab})
with Proposition \ref{momenty},ii) we see that 
\begin{eqnarray}
x_{k}\allowbreak &=&(-1)^{k}c_{k}/(k-1)!,  \label{x_c} \\
m_{n}(t) &=&n!P_{n}(c_{1}t,-c_{2}t,c_{3}t/2,\ldots ,(-1)^{n-1}tc_{n}/(n-1)!),
\label{mY}
\end{eqnarray}%
where $P_{n}$ is the mentioned above Yablonski's polynomial. In view of (\ref%
{x_c}) we see that the condition $\lim \sup_{k\longrightarrow \infty
}\left\vert x_{k}\right\vert ^{1/k}<\infty $ is equivalent to the following
one: $\lim \sup_{k\longrightarrow \infty }\left\vert c_{k}\right\vert
^{1/k}/k<\infty .$ However as Proposition 1.5 of \cite{Sim98} shows it can
happen that in the case of deterministic moment problem (i.e. when
coefficients $c_{k}$ fully determine distribution $dK)$ $\lim
\sup_{k\longrightarrow \infty }\left\vert c_{k}\right\vert ^{1/k}/k$ can be
finite or infinite. Hence existence of expansion (\ref{Yab}) has nothing to
do with determinacy of the L\'{e}vy process by its moments.

Following formulae (\cite{Yabl08},(2.2)--(2.4)) and using our notation given
by (\ref{mY}) we have the following properties of moments $m_{n}(t)$ which
we quote here for completeness of the paper:%
\begin{eqnarray}
m_{n+1}(t)\allowbreak &=&\allowbreak t\sum_{j=0}^{n}\binom{n}{j}%
c_{j+1}m_{n-j}(t),  \label{Y1} \\
\frac{\partial m_{n}(t)}{\partial c_{l}} &=&\left\{ 
\begin{array}{ccc}
0 & if & l>n \\ 
ntm_{n-l}(t) & if & l\leq n%
\end{array}%
\right. ,  \label{Y2} \\
m_{n}(t;\mathbf{c+d}) &=&\sum_{k=0}^{n}\binom{n}{k}m_{k}(t;\mathbf{c)}%
m_{n-k}(t;\mathbf{d),}  \label{Y3} \\
m_{n}(t;(c_{1}\alpha ,c_{2}\alpha ^{2},\ldots ))\allowbreak &=&\allowbreak
\alpha ^{n}m_{n}(t;(c_{1},c_{2},\ldots )).  \label{Y4}
\end{eqnarray}%
where we denoted $m_{n}(t;\mathbf{c)}$ $n-$th moment of the L\'{e}vy process
with parameters $\mathbf{c\allowbreak =\allowbreak (}c_{1},c_{2},\ldots ).$

Finally let us remark that as shown in \cite{Sol08} Yablonski's polynomials $%
P_{n}$ are closely related to the Kailath--Segall polynomials (see \cite%
{Segal76}) that are used to study the path properties of L\'{e}vy processes.
Hence our results give new interpretation of these polynomials.
\end{remark}

Using this formula and (\ref{Y1}) we have the following set of useful
relationships:

\begin{lemma}
\label{pom2}i) 
\begin{equation*}
M_{1}(x,t)M_{n}(x,t)\allowbreak =\allowbreak M_{n+1}(x,t)+t\sum_{k=1}^{n}%
\binom{n}{k}c_{k+1}M_{n-k}(x,t).
\end{equation*}
Thus in particular $EM_{1}(X_{t},t)M_{n}(X_{t},t)\allowbreak =\allowbreak
tc_{n+1}.$

ii) 
\begin{gather*}
M_{2}(x,t)M_{n}(x,t)\allowbreak =\allowbreak
M_{n+2}(x,t)+2nc_{2}tM_{n}(x,t)\allowbreak +\allowbreak t\sum_{k=2}^{n+1}(%
\binom{n}{k-1} \\
+2\binom{n}{k})c_{k+1}M_{n-k+1}(x,t)\allowbreak +\allowbreak
t^{2}\sum_{l=2}^{n}\binom{n}{l}M_{n-l}(x,t)\sum_{k=1}^{l-1}\binom{l}{k}%
c_{k+1}c_{l-k+1}.
\end{gather*}
In particular $EM_{2}(X_{t},t)M_{n}(X_{t},t)\allowbreak =\allowbreak
tc_{n+2}\allowbreak +\allowbreak t^{2}\sum_{k=1}^{n-1}\binom{n}{k}%
c_{k+1}c_{n+1-k}.$

iii) $\forall n,k\geq 0,t\geq 0:$ 
\begin{equation*}
E(M_{k}(X_{t},t)M_{n}(X_{t},t))\allowbreak =\allowbreak \allowbreak \left. 
\frac{\partial ^{n}\partial ^{k}}{\partial u^{n}\partial v^{k}}\exp
(t(f(u+v)-f(u)-f(v))\right\vert _{u=v=0},
\end{equation*}%
consequently 
\begin{equation*}
EM_{n}(X_{t},t)M_{k}(X_{t},t)\allowbreak =\allowbreak \sum_{j=1}^{\min
(k,n)}d_{j}^{(k,n)}t^{j},
\end{equation*}
with 
\begin{equation}
d_{j}^{(k,n)}\allowbreak =\left. \frac{d^{n+k-j}}{dx^{n+k-j}}\left(
h(x)\right) ^{j}\right\vert _{x=0},  \label{wsp}
\end{equation}%
where we denoted $h(x)=\sum_{k\geq 2}c_{k}x^{k-1}/(k-1)!\allowbreak
=\allowbreak f^{\prime }(x)-c_{1}.$ In particular coefficient by $t$ is
equal to $c_{n+k}$, by $t^{2}$ $\left. \frac{d^{n+k-2}}{dx^{n+k-j}}\left(
h(x)\right) ^{2}\right\vert _{x=0}$ and by $t^{\min (n,k)}$ is equal to $%
\left. \frac{d^{\max (n,k)}}{dx^{\max (n,k)}}\left( h(x)\right) ^{\min
(n,k)}\right\vert _{x=0}\allowbreak .$ If $n\allowbreak =\allowbreak k$
coefficient by $t^{k}\allowbreak $ is \allowbreak equal to $k!c_{2}^{k}>0,$ $%
.$
\end{lemma}

\begin{proof}
Rather tedious proof is shifted to Section \ref{dowody}.
\end{proof}

\subsection{Harnesses, reversed martingales and orthogonal polynomials\label%
{rev}}

As a immediate corollary we get the following nice property of the L\'{e}vy
processes

\begin{theorem}
\label{first}Let $\mathbf{X(\{}c_{1},c_{2},\ldots \})$ be some L\'{e}vy
process defined on $(0,\infty )$ and let $M_{1}(X_{t},t)$ be the first of
the polynomial martingales defined by Proposition (\ref{mart}).

Then $(M_{1}(X_{t},t)/t,\mathcal{F}_{\leq t})$ is the reversed martingale
and $M_{1}(X_{t},t)$ has the harness property that is:.%
\begin{eqnarray*}
\frac{1}{s}E(M_{1}(X_{s},s)|\mathcal{F}_{\geq t}) &=&\frac{1}{t}%
M_{1}(X_{t},t), \\
E(M_{1}(X_{t},t)|\mathcal{F}_{s,u})\allowbreak &=&\allowbreak \frac{u-t}{u-s}%
M_{1}(X_{s},s)+\frac{t-s}{u-s}M_{1}(X_{u},u),
\end{eqnarray*}%
where $s<t<u,$ and $\mathcal{F}_{s,u}\allowbreak =\allowbreak \sigma
(X_{v};v\in (0,s]\cup \lbrack u,\infty )).$
\end{theorem}

\begin{proof}
Simple proof strongly basing on Lemma \ref{pom2},i) is shifted to Section %
\ref{dowody}.
\end{proof}

Let us denote by $\left\{ Q_{j}(x,t)\right\} _{j\geq 0}$ system of monic
polynomials orthogonal with respect to marginal measure of $X_{t}$. By
assumption they are linearly independent and we have the following two
expansions:%
\begin{eqnarray*}
M_{n}(x,t)\allowbreak &=&\allowbreak \sum_{j=0}^{n}\hat{b}%
_{n,j}(t)Q_{j}(x,t), \\
Q_{n}(x,t)\allowbreak &=&\allowbreak \sum_{j=0}^{n}b_{n,j}(t)M_{j}(x,t).
\end{eqnarray*}%
We have the following simple observation:

\begin{proposition}
\label{ort}i) $\forall n\geq 1:b_{n,n}(t)\allowbreak =\allowbreak \hat{b}%
_{n,n}\left( t\right) \allowbreak =\allowbreak 1,$ $b_{n,0}(t)\allowbreak
=\allowbreak \hat{b}_{n,0}(t)\allowbreak =\allowbreak 0,$ hence in
particular: $Q_{1}(x,t)\allowbreak =\allowbreak M_{1}(x,t),$ $%
Q_{2}(x,t)\allowbreak =\allowbreak M_{2}(x,t)-c_{3}M_{1}(x,t)/c_{2},$

ii) $\forall n\geq 2:\hat{b}_{n,1}(t)\allowbreak =\allowbreak c_{n+1}/c_{2},$
hence in particular $M_{2}(x,t)\allowbreak =\allowbreak
Q_{2}(x,t)\allowbreak +\allowbreak c_{3}Q_{1}(x,t)/c_{2},$

iii) $\forall n\geq 2:\hat{b}_{n,2}(t)\allowbreak =\allowbreak
(tc_{2}\sum_{k=1}^{n-1}\binom{n}{k}c_{k+1}c_{n+1-k}\allowbreak +\allowbreak
c_{n+2}c_{2}\allowbreak -\allowbreak
c_{3}c_{n+1})/(2tc_{2}^{3}+c_{2}c_{4}-c_{3}^{2}).$

iv) The only L\'{e}vy process with all moments existing for whom polynomial
martingales $\left\{ M_{n}(X_{t},t)\right\} _{n\geq 0}$ are orthogonal is
the Wiener process with the variance equal to $c_{2}$.
\end{proposition}

\begin{proof}
i) Since we have both $EM_{n}(X_{t},t)\allowbreak =\allowbreak
EQ_{n}(X_{t},t)\allowbreak =\allowbreak 0$ for all $n\geq 1$ we deduce that
both $b_{n,0}(t)\allowbreak =\allowbreak \hat{b}_{n,0}(t)\allowbreak
=\allowbreak 0.$ Also since both systems of polynomials $\left\{
Q_{j}\right\} $ and $\left\{ M_{j}\right\} $ are monic then $\hat{b}%
_{n,n}(t)\allowbreak =\allowbreak b_{n,n}(t)\allowbreak =\allowbreak 1.$
Hence in particular $Q_{1}(x,t)\allowbreak =\allowbreak M_{1}(x,t).$

ii) On one hand by assertion i) of Lemma \ref{pom2} we have $%
EQ_{1}(X_{t},t)M_{n}(X_{t},t)\allowbreak =\allowbreak tc_{n+1}$ while by
assumption concerning polynomials $Q_{n}$ we get: $\hat{b}%
_{n,1}(t)EQ_{1}^{2}(X_{t},t)\allowbreak =\allowbreak \hat{b}_{n,1}(t)tc_{2}.$
Hence $\hat{b}_{n,1}(t)\allowbreak =\allowbreak c_{n+1}/c_{2}.$

iii) We have $Q_{2}(x,t)\allowbreak =\allowbreak
M_{2}(x,t)-c_{3}M_{1}(x,t)/c_{2}$ and consequently: $E(Q_{2}^{2}(X_{t},t))%
\allowbreak =\allowbreak tc_{4}+2t^{2}c_{2}^{2}\allowbreak -\allowbreak
2tc_{3}^{2}/c_{2}\allowbreak +\allowbreak
tc_{3}^{2}c_{2}/c_{2}^{2}\allowbreak =\allowbreak
t(2tc_{2}^{3}+c_{2}c_{4}-c_{3}^{2})/c_{2}$ and $%
EM_{n}(X_{t},t)Q_{2}(X_{t},t)\allowbreak =\allowbreak
tc_{n+2}+t^{2}\sum_{k=1}^{n-1}\binom{n}{k}c_{k+1}c_{n+1-k}\allowbreak
-\allowbreak c_{3}tc_{n+1}/c_{2}\allowbreak =\allowbreak
t(tc_{2}\sum_{k=1}^{n-1}\binom{n}{k}c_{k+1}c_{n+1-k}\allowbreak +\allowbreak
c_{n+2}c_{2}\allowbreak -\allowbreak c_{3}c_{n+1})/c_{2}.$

iv) One can see that condition $Q_{n}(x,t)\allowbreak =\allowbreak
M_{n}(x,t) $ for all $n\geq 1$ is satisfied by assertion ii) of Proposition %
\ref{ort} only if $c_{i}\allowbreak =\allowbreak 0$ for all $i\geq 3.$ On
the other hand for the Wiener process Hermite polynomials that generate
martingales by the formula $\left( c_{2}t\right) ^{n/2}H_{n}(x/\sqrt{c_{2}t}%
)\allowbreak =\allowbreak M_{n}(x,t)$ constitute also family of orthogonal
polynomials of the marginal distribution which is of course $N(0,c_{2}t).$
\end{proof}

Our main concern in this paper is to select those L\'{e}vy processes with
all moments existing that have also polynomial reversed martingales and
orthogonal martingales (that necessarily are also reversed martingales as
remarked in \cite{SzablPoly}, Corollary 5).

The problems that we will approach now are the following:

\begin{problem}
\label{part_pr}Fix $n.$ Can we find such rational (in $t$) function $\mu
_{n}(t)$, such that $\mu _{n}(t)M_{n}(t)$ is a reversed martingale.
\end{problem}

The next problem is a generalization of the above mentioned problem.

\begin{problem}
\label{gen_pr}Fix $n.$ Can we find such rational (in $t$) functions $\mu
_{k}(t)$, $i=1,3,\ldots ,n$ that%
\begin{equation}
R_{n}(X_{t},t)\allowbreak =\allowbreak \sum_{k=1}^{n}\mu
_{k}(t)M_{k}(X_{t},t),  \label{comb}
\end{equation}%
is a reversed martingale.
\end{problem}

\begin{remark}
As it can be easily noticed technically the reversed martingale property is
equivalent to the following condition: for all $0<s<t,$ $l\geq 1$ : 
\begin{equation*}
\mu _{n}(s)EM_{n}(X_{s},s)M_{l}(X_{s},s)=\mu
_{n}(t)EM_{n}(X_{t},t)M_{l}(X_{t},t)\allowbreak ,
\end{equation*}%
in case of Problem \ref{part_pr} and for all $0<s<t,$ $l\geq 1$ : 
\begin{equation}
\sum_{k=1}^{n}\mu _{k}(s)EM_{k}(X_{s},s)M_{l}(X_{s},s)\allowbreak
=\allowbreak \sum_{k=1}^{n}\mu _{k}(t)EM_{k}(X_{t},t)M_{l}(X_{t},t).
\label{og_r_m}
\end{equation}%
in the case of Problem \ref{gen_pr}.
\end{remark}

\begin{proof}
In case of Problem \ref{part_pr} we have $E(\mu _{n}(s)M_{n}(X_{s},s)|%
\mathcal{F}_{\geq t})\allowbreak =\allowbreak \mu _{n}(t)M_{n}(X_{t},t).$
Multiplying both sides by $M_{l}(X_{t},t)$ and taking expectation we get
right hand side while for the l-st we have $E(\mu
_{n}(s)M_{n}(X_{s},s)M_{l}(X_{t},t)\allowbreak =\allowbreak E(\mu
_{n}(s)M_{n}(X_{s},s)M_{l}(X_{s},s)$ since $\ M_{l}(X_{s},s)$ is the
martingale. The second case is treated similarly.
\end{proof}

We will solve the Problem \ref{part_pr} completely (Thm. \ref{gt3}) while
Problem \ref{gen_pr} only partially. Namely for $n\allowbreak =\allowbreak 2.
$ It is too complex to be solved in full generality in a short paper.

We will also consider the following simplified version of the above
mentioned general reversed martingale problem.

Namely we select those polynomial martingales $M_{n}(x,t)$ that multiplied
by some deterministic function $\mu _{n}(t)$ constitute a reversed
martingale.

One of our main result states that for $n\geq 3$ within the class of L\'{e}%
vy processes with all moments only the ones with all parameters $c_{i}$
equal to zero for $i\geq 3$ have this property.

First let us solve Problem \ref{gen_pr} for $n\allowbreak =\allowbreak 2.$

We have the following result;

\begin{theorem}
\label{glowne}Suppose that $\mathbf{X(\{}c_{1},c_{2},\ldots \})$ be some L%
\'{e}vy process defined on $(0,\infty ).$ Let $\left\{
M_{i}(X_{t},t)\right\} _{i\geq 1}$ be its polynomial martingales defined by (%
\ref{martingales}), then $\sum_{k=1}^{2}\mu _{k}(t)M_{k}(X_{t},t)$ is a
reversed martingale for some functions $\mu _{k}(t),$ $k\allowbreak
=\allowbreak 1,2$ iff functions $\mu _{1}(t)$ and $\mu _{2}(t)$ are the
following: 
\begin{eqnarray}
\mu _{2}(t)\allowbreak  &=&\allowbreak \frac{c_{2}-\beta c_{3}}{%
t(2c_{2}^{3}t+c_{2}c_{4}-c_{3}^{2})},  \label{_m1} \\
\mu _{1}(t) &=&\frac{\beta (2c_{2}^{2}t+c_{4})-c_{3}}{%
t(2c_{2}^{3}t+c_{2}c_{4}-c_{3}^{2})},  \label{_m2}
\end{eqnarray}%
where  $\beta $ is a constant and either of the following following cases
happen:

1) $c_{3}\allowbreak =\allowbreak 0,$ then 
\begin{equation}
\exp (tf(x)\allowbreak =\allowbreak e^{c_{1}tx}(\cos (x\sqrt{\frac{c_{4}}{%
2c_{2}}}))^{-2tc_{2}^{2}/c_{4}},  \label{momf}
\end{equation}%
for $\left\vert x\right\vert <\frac{\pi }{2}\sqrt{\frac{2c_{2}}{c_{4}}}.$ In
particular assuming for simplicity that $c_{1}\allowbreak =\allowbreak 0$
the distribution of $X_{t}$ for $t\allowbreak =\allowbreak \frac{c_{4}}{%
2c_{2}^{2}}$ has density $h(y)$ equal to 
\begin{equation}
h(y)\allowbreak =\allowbreak \frac{\sqrt{c_{4}}}{\sqrt{8c_{2}}\cosh (\frac{%
\pi y\sqrt{2c_{2}}}{2\sqrt{c_{4}}})};~y\in \mathbb{R}.  \label{gest}
\end{equation}%
and is identifiable by moments.

2) $c_{4}/c_{2}\allowbreak =\allowbreak c_{3}^{2}/c_{2}^{2}$ then L\'{e}vy
measure of such a process is degenerated, concentrated at $c_{3}/c_{2}$ and
consequently $\mathbf{X(\{}c_{1},c_{2},\ldots \})$ is in this case the
mixture of Poisson and Gaussian processes depending if $c_{3}\allowbreak
=\allowbreak c_{2}$ (pure Poisson case) or $c_{3}=c_{4}\allowbreak
=\allowbreak 0$ pure Gaussian case or $\frac{c_{3}}{c_{2}}\neq 0$ or $1$ the
nontrivial mixture.

3) $2c_{4}/c_{2}\allowbreak =\allowbreak c_{3}^{2}/c_{2}^{2}$ then 
\begin{equation*}
\exp (tf(x))\allowbreak =\allowbreak e^{(c_{1}-2c_{3}/c_{2})tx}\left( \frac{1%
}{1-c_{3}x/(2c_{2})}\right) ^{4tc_{3}^{2}/c_{2}^{2}}
\end{equation*}%
that is one dimensional distributions are of shifted gamma type.

4) $2c_{4}/c_{2}\allowbreak >\allowbreak c_{3}^{2}/c_{2}^{2},$ then 
\begin{gather}
\exp (tf(x))=\exp (xt(c_{1}-\frac{c_{3}c_{2}}{c_{4}c_{2}-c_{3}^{2}}))
\label{og1} \\
\times \left( \frac{1+\frac{\chi _{3}}{2\alpha }\tan (x\alpha )}{1-\frac{%
\chi _{3}}{2\alpha }\tan (x\alpha )}\frac{1}{2\alpha ^{2}-\chi
_{3}^{2}/2+(2\alpha ^{2}+\chi _{3}^{2}/2)\cos 2x\alpha }\right)
^{2t/(4\alpha ^{2}-\chi _{3}^{2})},  \label{og2}
\end{gather}%
where we denoted $\alpha \allowbreak =\allowbreak \frac{1}{2}\sqrt{2\frac{%
c_{4}}{c_{2}}-3\chi _{3}^{2}}$ and $\chi _{3}\allowbreak =\allowbreak
c_{3}/c_{2}.$

5) $2c_{4}/c_{2}\allowbreak <\allowbreak c_{3}^{2}/c_{2}^{2},$ then%
\begin{gather}
\exp (tf(x))=\exp (xt(c_{1}-\frac{c_{3}c_{2}}{c_{4}c_{2}-c_{3}^{2}}))
\label{og3} \\
\times \left( \frac{1+\frac{\chi _{3}}{2\alpha }\tanh (x\alpha )}{1-\frac{%
\chi _{3}}{2\alpha }\tanh (x\alpha )}\frac{1}{2\alpha ^{2}-\chi
_{3}^{2}/2+(2\alpha ^{2}+\chi _{3}^{2}/2)\cosh 2x\alpha }\right)
^{2t/(4\alpha ^{2}-\chi _{3}^{2})}.  \label{og4}
\end{gather}
\end{theorem}

\begin{proof}
is shifted to Section \ref{dowody}.
\end{proof}

Notice that even if both $\sum_{k=1}^{2}\mu _{k}(t)M_{k}(X_{t},t)$ and $%
M_{1}(X_{t},t)$ are the reversed martingales it does not mean that for some
function $\tilde{\mu}(t)$ $\tilde{\mu}(t)M_{2}(X_{t},t)$ is a reversed
martingale. As it will follow from the observations below the property that $%
\tilde{\mu}_{l}(t)M_{l}(X_{t},t)$ is a reversed martingale for some function 
$\tilde{\mu}_{l}(t)$ is somewhat independent from the property that linear
combination of martingales $M_{i},$ $i\allowbreak =\allowbreak 1,\ldots ,l$
(such as (\ref{comb})) is a reversed martingale.

It is so since we have the following observations.

\begin{lemma}
\label{pom}Let $\mathbf{X(\{}c_{1},c_{2},\ldots \})$ be L\'{e}vy process
defined on $(0,\infty )$ and let $\left\{ M_{n}(X_{t},t)\right\} _{n\geq 1}$
be polynomial martingales defined by (\ref{mart}). Suppose for $k\geq 2:$ $%
\mu (t)M_{k}(X_{t},t)$ is the reversed martingale, then \newline
i) for all $l\allowbreak =\allowbreak 1,2,\ldots $ 
\begin{equation}
\mu (s)EM_{l}(X_{s},s)M_{k}(X_{s},s)\allowbreak =\allowbreak \mu
(t)EM_{l}(X_{t},t)M_{k}(X_{t},t),  \label{war}
\end{equation}%
where $\mu (t)\allowbreak =\allowbreak 1/EM_{k}(X_{t},t)M_{k}(X_{t},t),$ ($%
EM_{l}(X_{t},t)M_{k}(X_{t},t)\allowbreak ,$ are given by Lemma \ref{pom2},
iii)),

ii) $\allowbreak c_{j}=\allowbreak 0,$ $j\allowbreak =\allowbreak \max
(3,k-1),\ldots ,2k-1.$
\end{lemma}

\begin{proof}
Is shifted to Section \ref{dowody}.
\end{proof}

\begin{remark}
Let us notice that polynomials $p_{k}(t)\allowbreak =\allowbreak
EM_{k}(X_{t},t)M_{k}(X_{t},t)$ are in fact the so called 'angular brackets'
of the polynomial martingales $M_{k}(X_{t},t).$ We know that they are
non-decreasing functions of $t$ and Lemma \ref{pom2}, iii) gives its precise
form.
\end{remark}

As an immediate corollary of Lemma \ref{pom},ii) and Remark \ref{Kolmog}%
,iii) we have the following result.

\begin{theorem}
\label{gt3}For $k\geq 3$ there does not exist function $\mu \left( t\right) $
such that $\mu (t)M_{k}(t)$ is a reversed martingale unless $%
c_{i}\allowbreak =\allowbreak 0$ for $i\geq 3.$
\end{theorem}

\begin{proof}
By Lemma \ref{pom} we know that parameters $c_{3},c_{4},\ldots ,c_{2k-1}$
are equal to zero. In particular we have $c_{4}\allowbreak =\allowbreak 0$
which leads by Remark \ref{Kolmog},iii) to the conclusion that $%
c_{i}\allowbreak =\allowbreak 0$ for $i\geq 3.$
\end{proof}

\begin{remark}
Notice that to have orthogonal polynomial martingales we have to have $%
EM_{l}(X_{s},s)M_{k}(X_{s},s)\allowbreak =\allowbreak 0$ for $k\neq l.$ The
presented above consideration show that it is possible only iff $%
c_{i}\allowbreak =\allowbreak 0$ for $i\geq 3.$ This corresponds with the
assertion iv) of the Proposition \ref{ort}.
\end{remark}

Thus it remains to consider the case $k\allowbreak =\allowbreak 2.$

\begin{remark}
\label{first2}The fact that $(\mu (t)M_{2}(X_{t},t),\mathcal{F}_{\leq t})$
is a reversed martingale implies by Lemma \ref{pom},ii) that $%
c_{3}\allowbreak =\allowbreak 0.$ Further from the proof of Theorem \ref%
{glowne} it follows that if $c_{3}\allowbreak =\allowbreak 0$ then $\mu
_{1}(t)\allowbreak =\allowbreak \beta /t$ and $\mu _{2}(t)\allowbreak
=\allowbreak $ $1/(2c_{2}^{2}t^{2}\allowbreak +\allowbreak
tc_{4})\allowbreak .$ Hence if $c_{3}\allowbreak =\allowbreak 0$ and $%
\sum_{i=1}^{2}\mu _{i}(t)M_{i}(X_{t},t)$ is a reverse martingale then $\mu
_{2}(t)M_{2}(X_{t},t)$ must also be a reversed martingale since $%
M_{1}(X_{t}t)/t$ is.
\end{remark}

\begin{remark}
\label{tan}Just for curiosity notice that it follows from Theorem \ref%
{glowne},1) that if $c_{1}\allowbreak =\allowbreak 0$ the moment generating
function of the process in this case is symmetric consequently that
coefficients $c_{j}$ with odd numbers are equal to zero and moreover numbers 
$\chi _{n}\allowbreak =\allowbreak c_{n}/c_{2}$ satisfy the following
recursion:%
\begin{equation}
\chi _{2(k+1)}\allowbreak =\allowbreak \frac{c_{4}}{2c_{2}}\sum_{j=0}^{k-1}%
\binom{2k}{2j+1}\chi _{2(j+1)}\chi _{2(k-j)},  \label{rekursja}
\end{equation}%
which after denoting $T_{j}\allowbreak =\allowbreak \chi _{2j}(\frac{2}{\chi
_{4}})^{j-1}$ can be reduced to the following one: 
\begin{equation}
T_{k+1}\allowbreak =\allowbreak \sum_{j=1}^{k}\binom{2k}{2k-1}T_{j}T_{k-j+1}.
\label{TnaT}
\end{equation}%
Little reflections shows that numbers $T_{k}$ are the so called tangent
numbers\footnote{%
seq A000182 on http://oeis.org} which surprisingly come to the L\'{e}vy
processes scene.
\end{remark}

\begin{remark}
As a corollary we can now refer to the third martingale characterization of
the Wiener process done by Weso\l owski in \cite{Wes90}. It states that if a
square integrable process $\mathbf{X\allowbreak =\allowbreak (}X_{t})_{t\geq
0}$ has the property that $(X_{t},\mathcal{F}_{\leq t})$ and $(X_{t}^{2}-t,%
\mathcal{F}_{\leq t})$ are martingales and $(X_{t}/t,\mathcal{F}_{\geq t})$
and $((X_{t}^{2}-t)/t^{2},\mathcal{F}_{\geq t})$ are reversed martingales
then the process is a Wiener process. It was shown in \cite{Szab-OU-W} that
this is not true characterization. Namely a counterexample with dependent
increments was shown.

If we however we confine ourselves to the class of L\'{e}vy processes having
all moments then this characterization is true. Since as shown above for our
class of L\'{e}vy processes with $c_{1}\allowbreak =\allowbreak 0,$ $%
c_{2}\allowbreak =\allowbreak 1,$ $(X_{t},\mathcal{F}_{\leq t})$ and $%
(X_{t}^{2}-t,\mathcal{F}_{\leq t})$ are martingales and $(X_{t}/t,\mathcal{F}%
_{\geq t})$ is the reversed martingale only condition that $%
((X_{t}^{2}-t)/t^{2},\mathcal{F}_{\geq t})$ is a reversed martingale
matters. Comparing this requirement with Theorem \ref{first2} we see that we
must have $c_{4}\allowbreak =\allowbreak 0$ to fulfill the requirement. But $%
c_{4}\allowbreak =\allowbreak 0$ leads to $c_{i}\allowbreak =\allowbreak 0,$
for all $i\geq 3$ by Remark \ref{Kolmog},iii).
\end{remark}

\section{Open problems\label{open}}

First of all let us ask the following general question. Theorem \ref{first}
was proved under assumption that we deal with the L\'{e}vy process with all
moments existing. The proof of this result was simple because it strongly
depended on this assumption.

\begin{problem}
Can we weaken this assumption? That is can we prove assertions of Theorem %
\ref{first} assuming that say the L\'{e}vy process has only first $m$ ($m$
some fixed integer) moments? Can we prove harness property of $M_{1}$
assuming only existence of the first $m$ moments and say knowing that $%
E(X_{s}|\mathcal{F}_{\geq t})$ for $s<t$ is a linear function of $X_{t}?$
\end{problem}

Let us return the general 'reversed martingale' Problem \ref{gen_pr}.

The case $n\allowbreak =\allowbreak 2$ was examined in Theorem \ref{glowne}.

\begin{problem}
What about $n>2$ can we find such functions $\mu _{k}(t),$ $k\allowbreak
=\allowbreak 1,\ldots ,n$ that $R_{n}$ (defined by \ref{comb}) is the
reversed martingale?
\end{problem}

Similarly one can pose the following problem concerning the so called
quadratic harnesses among L\'{e}vy process the problem inclusively studied
recently by Bryc, Weso\l owski and Matysiak (see \cite{BRWE12},\cite{BryBo},%
\cite{BryMaWe07}).

\begin{problem}
Find all L\'{e}vy process (i.e. coefficients $c_{i},$ $i\geq 3)$ such that $%
M_{2}(X_{t},t)$ is a quadratic harness i.e. 
\begin{gather*}
E(M_{2}(X_{t},t)|\mathcal{F}_{s,u})\allowbreak =\allowbreak
AM_{2}(X_{s},s)+BM_{1}(X_{s},s)M_{1}(X_{u},u) \\
+CM_{2}(X_{u},u)+DM_{1}(X_{s},s)+EM_{1}(X_{u},u)-c_{2}s^{2},
\end{gather*}%
where $0<s<t<u,$ $A,$ $B$, $C,$ $D,$ $E$ are some functions of $s,t,u$ only.
Note that $A,$ $B$, $C,$ $D,$ $E$ can be relatively easily found by solving
system of $5$ linear equations obtained by multiplying the above equality by 
$M_{2}(X_{s},s),$ $M_{2}(X_{u},u),$ $M_{1}(X_{s},s)M_{1}(X_{u},u),$ $%
M_{1}(X_{s},s)$ and $M_{1}(X_{u},u)$ and calculating expectation of both
sides and utilizing the fact that $M_{i}(X_{t},t),$ $i=1,2$ are martingales
(as done in \cite{SzablPoly}). Having $A,$ $B$, $C,$ $D,$ $E$ we multiply
both sides of this equality by $M_{l}(X_{s},s)M_{k}(X_{u},u)$ and calculate
their expectations. On the way we use Lemma \ref{pom2}i)-ii) and Lemma \ref%
{pom},i). In this way we get system of recursions to be satisfied by
coefficients $c_{i},$ $i\geq 4.$
\end{problem}

\begin{problem}
What about extension of these results to processes with nonhomogeneous,
independent increments. A stem in this direction is done in \cite{SolUz08t}.
\end{problem}

\section{Proofs\label{dowody}}

\begin{proof}[Proof of Proposition \protect\ref{momenty}]
i) We have $m_{n}(s+t)\allowbreak =\allowbreak EX_{t+s}^{n}\allowbreak
=\allowbreak E(X_{t+s}-X_{s}+X_{s})^{n}\allowbreak =\allowbreak
\sum_{j=0}^{n}\binom{n}{j}m_{j}(t)m_{n-j}(s)$ since $E(X_{t+s}-X_{s})^{n}%
\allowbreak =\allowbreak m_{n}(t)$ for the L\'{e}vy processes.

ii) Let us define $Q(t;x)\allowbreak =\allowbreak \sum_{j\geq
0}m_{n}(t)x^{j}/j!$ . Following i) we get 
\begin{equation*}
Q(t+s;x)\allowbreak =\allowbreak Q(t;x)Q(s;x).
\end{equation*}%
Since for fixed $x$ function the $Q$ is continuous in the first argument by
assumption we are dealing with multiplicative Cauchy equation. Hence $%
Q(t;x)\allowbreak =\allowbreak \exp (tf(x))$ for some constant $f(x)$
depending on $x.$ Since $Q(t;x)$ is analytic with respect to $x$ and also
since $Q(t;0)\allowbreak =\allowbreak 1$ we can expand function $f$ in a
power series of the form $f(x)\allowbreak =\allowbreak \sum_{k\geq
1}c_{k}x^{k}/k!.$ Following definition of the function $Q$ we get further
statements of ii).

iii) We have by direct calculation: $m_{1}(t)\allowbreak =\allowbreak
EX_{t}\allowbreak =\allowbreak \left. \frac{\partial }{\partial x}\exp
(tf(x))\right\vert _{x=0}\allowbreak =\allowbreak c_{1}t$ and $%
m_{2}(t)\allowbreak =\allowbreak EX_{t}^{2}\allowbreak =\allowbreak \left. 
\frac{\partial ^{2}}{\partial x^{2}}\exp (tf(x))\right\vert
_{x=0}\allowbreak =\allowbreak c_{1}^{2}t\allowbreak +\allowbreak c_{2}t.$
Now let us consider sequence $\hat{m}_{n}(t).$ We have $\hat{m}%
_{n}(t)\allowbreak =\allowbreak \sum_{i=0}^{n}\binom{n}{i}%
m_{n-i}(t)(-1)^{i}\left( c_{1}t\right) ^{i}$ and also $\sum_{i\geq 0}\left(
-1\right) ^{i}(c_{1}t)^{i}\frac{x^{i}}{i!}\allowbreak =\allowbreak \exp
(-c_{1}tx).$ Hence $\sum_{j\geq 0}\hat{m}_{n}(t)\frac{x^{n}}{n!}\allowbreak
=\allowbreak \exp (tf(x)-c_{1}tx)).$

iv) First of all let us notice that following definition of the function $Q$
we have $m_{n}^{^{\prime }}(t)\allowbreak =\allowbreak \left. \frac{\partial
^{n}\partial Q(t;x)}{\partial x^{n}\partial t}\right\vert _{x=0}\allowbreak
=\allowbreak \left. \frac{\partial ^{n}}{\partial x^{n}}(f(x)\exp
(tf(x))\right\vert _{x=0}.$ Now we apply Leibnitz formula for $n-$th
derivative of the product of two differentiable functions. On the way we
have to remember that $\left. \frac{d^{n}}{dx_{n}}f(x)\right\vert
_{x=0}\allowbreak =\allowbreak c_{n}.$

v) We have: 
\begin{gather*}
\sum_{n=0}^{\infty }\frac{u^{n}}{n!}\sum_{j=0}^{n}\binom{n}{j}%
m_{n-j}(-s)m_{j+i}(s)\allowbreak =\allowbreak \\
\exp (-sf(x))\sum_{j=0}^{\infty }\frac{u^{j}}{j!}m_{j+i}(s)\allowbreak
=\allowbreak \exp (-sf(u))\frac{\partial ^{i}}{\partial u^{i}}\exp (sf(u)).
\end{gather*}
\end{proof}

\begin{proof}[Proof of Lemma \protect\ref{pom2}]
i) First observe that $M_{1}(x,t)\allowbreak \mathcal{N}_{t}(x,r)=%
\allowbreak (x-c_{1}t)\mathcal{N}_{t}(x,r)\allowbreak =\allowbreak \frac{%
\partial }{\partial r}\mathcal{N}_{t}(x,r)+t(f^{\prime }(r)-c_{1})\mathcal{N}%
_{t}(x,r),$ where $\mathcal{N}_{t}(x,r)$ is the defined in Proposition \ref%
{mart}, ii) characteristic functions of polynomials $M_{n}.$ Hence using
Leibnitz's rule we get: 
\begin{gather*}
M_{1}(x,t)M_{n}(x,t)\allowbreak =\allowbreak \left. \frac{\partial ^{n}}{%
\partial r^{n}}M_{1}(x,t)\allowbreak \mathcal{N}_{t}(x,r)\right\vert _{r=0}
\\
=\allowbreak \left. \frac{\partial ^{n+1}}{\partial r^{n+1}}\mathcal{N}%
_{t}(x,r)\right\vert _{r=0}+t\left. \sum_{j=0}^{n}\binom{n}{j}\frac{\partial
^{j}}{\partial r^{j}}(f^{\prime }(r)-c_{1})\frac{\partial ^{n-j}}{\partial
r^{n-j}}\mathcal{N}_{t}(x,r)\right\vert _{r=0} \\
=M_{n+1}(x,t)+t\sum_{j=0}^{n}\binom{n}{j}c_{j+1}M_{n-j}(x,t)\allowbreak ,
\end{gather*}%
since obviously $\left. \frac{\partial ^{k}}{\partial r^{k}}\mathcal{N}%
_{t}(x,r)\right\vert _{r=0}\allowbreak =\allowbreak M_{k}(x,t).$

ii) Recall that $M_{2}(x,t)\allowbreak =\allowbreak M_{1}(x,t)^{2}-c_{2}t,$
hence using i) we get 
\begin{gather*}
M_{1}(x,t)^{2}M_{n}(x,t)=M_{1}(x,t)M_{n+1}(x,t)+t\sum_{k=1}^{n}\binom{n}{k}%
c_{k+1}M_{n-k}(x,t)M_{1}(x,t) \\
=M_{n+2}(x,t)+t\sum_{k=1}^{n+1}\binom{n+1}{k}c_{k+1}M_{n+1-k}(x,t)+t%
\sum_{k=1}^{n}\binom{n}{k}c_{k+1}M_{n-k+1}(x,t)+ \\
t^{2}\sum_{k=1}^{n-1}\binom{n}{k}c_{k+1}\sum_{j=1}^{n-k}\binom{n-k}{j}%
c_{j+1}M_{n-k-j}(x,t) \\
=M_{n+2}(x,t)+t\sum_{k=1}^{n+1}(\binom{n}{k-1}+2\binom{n}{k}%
)c_{k+1}M_{n-k+1}(x,t)+ \\
t^{2}\sum_{l=2}^{n}\binom{n}{l}M_{n-l}(x,t)\sum_{k=1}^{l-1}\binom{l}{k}%
c_{k+1}c_{l-k+1}
\end{gather*}%
Since $\binom{n+1}{k}\allowbreak =\allowbreak \binom{n}{k}\allowbreak
+\allowbreak \binom{n}{k-1}.$

iii) Notice that on one hand $EM_{n}(X_{t},t)M_{k}(X_{t},t)$ that is equal
to $E\left. \frac{\partial ^{n}\partial ^{k}}{\partial u^{n}\partial v^{k}}%
\mathcal{N}_{t}(X_{t},u)\mathcal{N}_{t}(X_{t},v)\right\vert
_{u=v=0}\allowbreak =\allowbreak $ $\left. \frac{\partial ^{n}\partial ^{k}}{%
\partial u^{n}\partial v^{k}}E\mathcal{N}_{t}(X_{t},u)\mathcal{N}%
_{t}(X_{t},v)\right\vert _{u=v=0}\allowbreak .\allowbreak $Now notice that $E%
\mathcal{N}_{t}(X_{t},u)\mathcal{N}_{t}(X_{t},v)\allowbreak =\allowbreak
E\exp ((u+v)X_{t}-t(f(u)+f(v))\allowbreak =\allowbreak \exp
(t(f(u+v)-f(u)-f(v))$ by Proposition \ref{mart}, ii). Notice that $\left. 
\frac{\partial ^{k}}{\partial v^{k}}\exp (t(f(u+v)-f(u)-f(v))\right\vert
_{u=0}\allowbreak =\allowbreak 0$ for $k\geq 1.$ Further notice that $\frac{%
\partial ^{k}}{\partial v^{k}}\exp (t(f(u+v)-f(u)-f(v))$ is a product of two
expressions : first being a polynomial in $t$ of order $k$ with coefficients
being some differential expressions of $f(u+v)-f(v)$ and the second $\exp
(t(f(u+v)-f(u)-f(v)).$ Consequently upon applying Leibnitz rule to this
product and setting $u=v=0$ we see that only the first expression matters.
The assertion follows the fact that $\left. \frac{\partial ^{n}}{\partial
u^{n}}(f^{(j)}(u+v)-f^{(j)}(v))\right\vert _{u=v=0}\allowbreak =\allowbreak
\left. \frac{\partial ^{n}}{\partial u^{n}}(\left.
f^{(j)}(u+v)-f^{(j)}(v)\right\vert _{v=0})\right\vert _{u=0}$ for $%
j\allowbreak =\allowbreak 1,\ldots ,k.$ Firstly we observe that $n-$th
derivative of $\exp (tf(x))$ with respect to $x$ is of the form $(tf^{\left(
n\right) }(x)\allowbreak +\allowbreak \ldots \allowbreak +\allowbreak
t^{n}(f^{\prime }(x))^{n})\exp (tf(x)).$ The independence of $c_{1}$ follows
the fact that $\exp (t(f(u+v)-f(u)-f(v))\allowbreak =\allowbreak \exp
(t(f(u+v)-c_{1}(u+v)-(f(u)-c_{1}u)-(f(v)-c_{1}v)))$ hence does not depend on 
$c_{1}.$ Thus visibly $EM_{n}(X_{t},t)M_{k}(M_{t},t)$ is a polynomial in $t$
of order $\min (n,k)$ with coefficient by $t^{j}$ equal to $\left. \frac{%
d^{n+k-j}}{dx^{n+k-j}}\left( f^{\prime }(x)\right) ^{j}\right\vert _{x=0}$
for $j\allowbreak =\allowbreak 0,\ldots ,\min (n,k).$
\end{proof}

\begin{proof}[Proof of Theorem \protect\ref{first}]
To see that $M_{1}(X_{t},t)/t$ is a reversed martingale we have to show that
for all $s<t$ and $l$ we have :%
\begin{equation*}
\frac{1}{s}EM_{1}(X_{s},s)M_{l}(X_{s},s)\allowbreak =\allowbreak \frac{1}{t}%
EM_{1}(X_{t},t)M_{l}(X_{t},t).
\end{equation*}%
By Lemma \ref{pom2},i) we see that this is satisfied.

To prove the 'harness' part we have to show for example that for all $r\in 
\mathbb{R}$. 
\begin{eqnarray*}
E\mathbb{\mathcal{N}}_{s}(X_{s},r)M_{1}(X_{t},t)\mathcal{N}%
_{u}(X_{u},r)\allowbreak &=&\allowbreak \frac{u-t}{u-s}E\mathcal{N}%
_{s}(X_{s},r)M_{1}(X_{s},s)\mathcal{N}_{u}(X_{u},r) \\
&&+\frac{t-s}{u-s}E\mathcal{N}_{s}(X_{s},r)M_{1}(X_{u},u)\mathcal{N}%
_{u}(X_{u},r).
\end{eqnarray*}%
Recall that $\mathcal{N}_{t}(x,r)\allowbreak =\allowbreak \exp (tx-tf(r))$ .
Utilizing martingale property of $\mathcal{N}_{t}(X_{t},r)$ we get: \newline
\begin{gather*}
E\mathbb{\mathcal{N}}_{s}(X_{s},r)M_{1}(X_{t},t)\mathcal{N}%
_{u}(X_{u},r)\allowbreak =\allowbreak \allowbreak E\mathbb{\mathcal{N}}%
_{s}(X_{s},r)M_{1}(X_{t},t)\mathcal{N}_{t}(X_{t},r)\allowbreak \allowbreak \\
=\allowbreak \frac{\partial }{\partial r}E\left( \mathbb{\mathcal{N}}%
_{s}(X_{s},r)\right) ^{2}+t(f^{\prime }(r)-c_{1})E\left( \mathbb{\mathcal{N}}%
_{s}(X_{s},r)\right) ^{2}.
\end{gather*}%
By the similar argument we have: 
\begin{equation*}
E\mathcal{N}_{s}(X_{s},r)M_{1}(X_{s},s)\mathcal{N}_{u}(X_{u},r)\allowbreak
=\allowbreak \frac{\partial }{\partial r}E\left( \mathbb{\mathcal{N}}%
_{s}(X_{s},r)\right) ^{2}+s(f^{\prime }(r)-c_{1})E\left( \mathbb{\mathcal{N}}%
_{s}(X_{s},r)\right) ^{2}
\end{equation*}%
and 
\begin{equation*}
E\mathcal{N}_{s}(X_{s},r)M_{1}(X_{u},u)\mathcal{N}_{u}(X_{u},r)\allowbreak
=\allowbreak \frac{\partial }{\partial r}E\left( \mathbb{\mathcal{N}}%
_{s}(X_{s},r)\right) ^{2}+u(f^{\prime }(r)-c_{1})E\left( \mathbb{\mathcal{N}}%
_{s}(X_{s},r)\right) ^{2}.
\end{equation*}%
The desired equality follows since $\frac{u-t}{u-s}\allowbreak +\allowbreak 
\frac{t-s}{u-s}\allowbreak =\allowbreak 1$ and $t\allowbreak =\allowbreak 
\frac{u-t}{u-s}s+u\frac{t-s}{u-s}.$
\end{proof}

\begin{proof}[Proof of Lemma \protect\ref{pom}]
First of all notice that if $\mu (t)M_{k}(t)$ is a reversed martingale then $%
E(\mu (s)M_{k}(X_{s},s)|\mathcal{F}_{\geq t})\allowbreak =\allowbreak \mu
(t)M_{k}(X_{t},t)$ a.s., hence multiplying both sides by $M_{l}(X_{t},t)$
and taking expectation of both sides we get $\mu
(s)EM_{k}(X_{s},s)M_{l}(X_{t},t)\allowbreak =\allowbreak \mu
(t)EM_{k}(X_{t},t)M_{l}(X_{t},t).$ Finally we use the fact that $M_{l}$ is a
martingale. Thus we get (\ref{war}). By Lemma \ref{pom2},iii) we know that $%
EM_{k}(X_{t},t)M_{l}(X_{t},t)$ is a polynomial of order $\min (k,l)$ in $t$.
Moreover if $l\allowbreak =\allowbreak k$ coefficient by $t^{k}$ is equal to 
$k!c_{2}^{k}>0.$ Secondly notice that quantity $\mu
(t)EM_{k}(X_{t},t)M_{l}(X_{t},t)$ has to be independent on $t,$ thus since
for $l\allowbreak =\allowbreak k$ $EM_{k}(X_{t},t)M_{l}(X_{t},t)$ is a
polynomial in $t$ of exactly $k-$th order we deduce that $\mu (t)$ must be
proportional to the inverse of $EM_{k}(X_{t},t)M_{k}(X_{t},t).$

i) By Lemma \ref{pom2},iii) we know that for $l<k$ $%
EM_{k}(X_{t},t)M_{l}(X_{t},t)$ is a polynomial in $t$ of order $l,$ so if $%
\mu (t)EM_{k}(X_{t},t)M_{l}(X_{t},t)$ is to be independent of $t$ $%
EM_{k}(X_{t},t)M_{l}(X_{t},t)$ must be zero polynomial.

ii) The fact that $c_{k+l}\allowbreak =\allowbreak 0,$ $l\allowbreak
=\allowbreak 1,\ldots ,k-1$ follows formula $d_{1}^{(k,l)}\allowbreak
=\allowbreak c_{k+l}$ and the fact that $EM_{k}(X_{t},t)M_{l}(X_{t},t)$ for $%
l<k$ must be zero polynomial in particular its coefficients by $t$ (which
are equal to $c_{k+l})$ must be equal to zero. In this way we get the case $%
k\allowbreak =\allowbreak 2.$ Let us now consider coefficient in $%
EM_{k}(X_{t},t)M_{l}(X_{t},t)$ by $t^{2}.$ It is equal to $\sum_{j=1}^{l+k-3}%
\binom{l+k-2}{j}c_{j+1}c_{l+k-j-1}$ as indicated by Lemma \ref{pom2},ii).
Let us now take into account the fact that $c_{k+1},\ldots ,c_{2k-1}$ are
equal to zero. It means that in fact we have to have: $\sum_{j=l-1}^{k-1}%
\binom{l+k-2}{j}c_{j+1}c_{l+k-j-1}\allowbreak =\allowbreak 0.$ Now we change
index of summation to $s\allowbreak =\allowbreak j-l+1$ and get that for all 
$l\allowbreak =\allowbreak 2,\ldots ,k-1$ we have to have $\sum_{s=0}^{k-l}%
\binom{k+l-2}{s+l-1}c_{s+l}c_{k-s}\allowbreak =\allowbreak 0.$ Let us
consider $l\allowbreak =\allowbreak k-1$ and $k-2.$ From the first equality
we deduce that $c_{k}c_{k-1}\allowbreak =\allowbreak 0$ and from the second
that $(\binom{2k-4}{k-3}\allowbreak +\binom{2k-4}{k-1})\allowbreak
c_{k-2}c_{k}\allowbreak +\allowbreak \binom{2k-4}{k-2}c_{k-1}^{2}\allowbreak
=\allowbreak 0.$ Now if $k\allowbreak =\allowbreak 3$ and $c_{2}>0$ we
deduce that $c_{3}\allowbreak =\allowbreak 0$ when $k\allowbreak
=\allowbreak 3.$ Thus let us take $k\geq 4.$ By multiplying both sides of
the last equality by $c_{k-1}$ we deduce that since $c_{k-1}c_{k}\allowbreak
=\allowbreak 0$ that $c_{k-1}\allowbreak =\allowbreak 0,$ or equivalently
that $c_{k-2}c_{k}\allowbreak =\allowbreak 0$. Let us consider now $%
l\allowbreak =\allowbreak k-3.$ We get $((\binom{2k-5}{k-4}\allowbreak
+\allowbreak \binom{2k-5}{k-1})c_{k-3}c_{k}\allowbreak +\allowbreak (\binom{%
2k-5}{k-3}\allowbreak \allowbreak +\allowbreak \binom{2k-5}{k-2}%
)c_{k-2}c_{k-1}\allowbreak =\allowbreak 0.$ Hence $c_{k}c_{k-3}\allowbreak
=\allowbreak 0$ and so on. But after $k-2$ such steps we will get $%
c_{2}c_{k}\allowbreak =\allowbreak 0.$ But $c_{2}>\allowbreak 0.$ So we
deduce that $c_{k}\allowbreak =\allowbreak 0.$
\end{proof}

\begin{proof}[Proof of Theorem \protect\ref{glowne}]
First of all let us notice that condition (\ref{og_r_m}) for $l\allowbreak
=\allowbreak 1,2$ leads to the following two linear equations:%
\begin{eqnarray*}
\mu _{1}EM_{1}^{2}(X_{t}.t)+\mu _{2}EM_{1}(X_{t},t)M_{2}(X_{t},t)\allowbreak
&=&\allowbreak \beta , \\
\mu _{1}EM_{1}(X_{t}.t)M_{2}(X_{t},t)+\mu _{2}EM_{2}^{2}(X_{t},t)\allowbreak
&=&\allowbreak 1
\end{eqnarray*}%
for the functions $\mu _{1}$ and $\mu _{2}.$ Hence indeed they are given by (%
\ref{_m1}) and (\ref{_m2}). Now notice that in order that say right hand
side of (\ref{og_r_m}) be independent on $t$ for $l\geq 3$ we have to have:%
\begin{eqnarray*}
-2\beta c_{2}^{2}c_{l+1}+(\beta c_{3}-c_{2})\sum_{k=1}^{l-1}\binom{l}{k}%
c_{k+1}c_{n+1-k})) &=&2A_{l}c_{2}^{3}, \\
(c_{3}-\beta c_{4})c_{l+1}+(-c_{2}+\beta c_{3})c_{l+2}
&=&A_{l}(c_{2}c_{4}-c_{3}^{2}),
\end{eqnarray*}%
for some constant $A_{l}.$ Let us denoting $\chi _{l}\allowbreak
=\allowbreak c_{l}/c_{2}$ eliminate $A_{l}$ from the above equations. We
will get then:%
\begin{equation*}
(\chi _{4}-\chi _{3}^{2})(-\beta \chi _{l+1}+\frac{(\beta \chi _{3}-1)}{2}%
\sum_{k=1}^{l-1}\binom{l}{k}\chi _{k+1}\chi _{n+1-k})=\chi _{l+2}(\beta \chi
_{3}-1)+\chi _{l+1}(\chi _{3}-\beta \chi _{4}).
\end{equation*}%
This equation is equivalent to the following relationship:%
\begin{equation*}
(1-\beta \chi _{3})(\chi _{l+2}-\chi _{l+1}\chi _{3}-\frac{(\chi _{4}-\chi
_{3}^{2})}{2}\sum_{k=1}^{l-1}\binom{l}{k}\chi _{k+1}\chi _{l+1-k})=0
\end{equation*}%
$(1-\beta \chi _{3})=0$ leads to $\mu _{2}\allowbreak =\allowbreak 0$ so let
us assume that%
\begin{equation*}
\chi _{l+2}\allowbreak =\allowbreak \chi _{3}\chi _{l+1}+\frac{(\chi
_{4}-\chi _{3}^{2})}{2}\sum_{k=1}^{l-1}\binom{l}{k}\chi _{k+1}\chi _{l+1-k}.
\end{equation*}%
Let us denote $\varphi (r)\allowbreak =\allowbreak \sum_{n=2}^{\infty }\frac{%
r^{n-2}}{(n-2)^{2}}\chi _{n}$. Comparing this definition with (\ref{kolmog})
we see that $\varphi (r)\allowbreak =\allowbreak f^{\prime \prime }(r)/c_{2}.
$ Notice that $\varphi \left( 0\right) \allowbreak =\allowbreak 1.$
Multiplying both sides by $\frac{r^{l-1}}{(l-1)!}$ and sum by $l$ from $1$
to $\infty .$ 
\begin{eqnarray*}
\varphi ^{\prime }(r)\allowbreak  &=&\allowbreak \chi _{3}\varphi (r)+\frac{%
(\chi _{4}-\chi _{3}^{2})}{2}\sum_{l=2}^{\infty }\frac{r^{l-1}}{(l-1)!}%
\sum_{k=1}^{l-1}\binom{l}{k}\chi _{k+1}\chi _{l+1-k} \\
&=&\chi _{3}\varphi (r)+\frac{(\chi _{4}-\chi _{3}^{2})}{2}%
\sum_{k=1}^{\infty }\frac{r^{k-1}}{k!}\chi _{k+1}\sum_{l=k+1}^{\infty }\frac{%
lr^{l-k}}{(l-k)!}\chi _{l+1-k} \\
&=&\chi _{3}\varphi (r)+\frac{(\chi _{4}-\chi _{3}^{2})}{2}%
\sum_{k=1}^{\infty }\frac{r^{k-1}}{k!}\chi _{k+1}\sum_{m=1}^{\infty }\frac{%
(k+m)r^{m}}{m!}\chi _{m+1} \\
&=&\chi _{3}\varphi (r)+\frac{(\chi _{4}-\chi _{3}^{2})}{2}%
\sum_{k=1}^{\infty }\frac{r^{k-1}}{k!}\chi _{k+1}(k\sum_{m=1}^{\infty }\frac{%
r^{m}}{m!}\chi _{m+1}+\sum_{m=1}^{\infty }\frac{r^{m}}{(m-1)!}\chi _{m+1}) \\
&=&\chi _{3}\varphi (r)+\frac{(\chi _{4}-\chi _{3}^{2})}{2}2\varphi
(r)\int_{0}^{r}\varphi (x)dx.
\end{eqnarray*}%
So we have end up with the following differential equation:%
\begin{equation}
\psi ^{\prime \prime }(r)-\chi _{3}\psi ^{\prime }(r)-v\psi ^{\prime
}(r)\psi (r)=0  \label{diff}
\end{equation}%
where we denoted $\psi (r)\allowbreak =\allowbreak \int_{0}^{r}\varphi
(x)dx\allowbreak =\allowbreak f^{\prime }(r)/c_{2},$ $v\allowbreak
=\allowbreak (\chi _{4}-\chi _{3}^{2})$ with initial conditions $\psi
(0)\allowbreak =\allowbreak 0,$ $\psi ^{\prime }(0)\allowbreak =\allowbreak
1.$ Before solving this equation in full generality let us consider
particular cases.

\begin{enumerate}
\item Let us assume that $v>0$ and $\chi _{3}\allowbreak =\allowbreak 0.$
Our equation now becomes 
\begin{equation*}
\psi ^{\prime \prime }(r)-v\psi ^{\prime }(r)\psi (r)=0,
\end{equation*}%
which leads to $\psi (r)=\sqrt{\frac{2C_{1}}{v}}\tan (\sqrt{\frac{C_{1}}{2v}}%
(r+C_{2})).$ Taking into account initial conditions we get $\psi (r)=\sqrt{%
\frac{2}{\chi _{4}}}\tan (\frac{r}{\sqrt{2\chi _{4}}})$ and consequently
recalling that $\int \tan (ax)dx\allowbreak =\allowbreak -\log \cos (ax)/a$
we get (\ref{momf}).

\item $v\allowbreak =\allowbreak 0$ or $c_{4}/c_{2}\allowbreak =\allowbreak
c_{3}^{2}/c_{2}^{2}$ which means that (recalling interpretation of
coefficients $c_{n}$ presented in Remark \ref{Kolmog}) variance of the L\'{e}%
vy measure of our process is equal to zero consequently that L\'{e}vy is
degenerated and concentrated at the point $c_{3}/c_{2}.$ In this case
equation (\ref{diff}) is reduced to the following: $\psi ^{\prime \prime
}(r)-\chi _{3}\psi ^{\prime }(r)=0$ which gives (after taking into account
initial conditions) $\psi (r)\allowbreak =\allowbreak \exp (r\chi _{3})/\chi
_{3}.$ Hence we get the assertion.

\item If $v\allowbreak =\allowbreak \chi _{3}^{2}/2$ then one can easily
check that the following function: 
\begin{equation*}
\psi (x)=\frac{x}{1-\chi _{3}x/2}
\end{equation*}%
satisfies conditions $\psi (0)=0$ and $\psi ^{\prime }(0)\allowbreak
=\allowbreak 1$ and moreover satisfies differential equation (\ref{diff})
with $v\allowbreak =\chi _{3}^{2}/2.$ Hence $f(x)\allowbreak =\allowbreak
c_{1}x+\frac{-2xc_{2}}{c_{3}}-4c_{2}^{2}\ln (1-xc_{3}/(2c_{2}))/c_{3}^{2}$

\item If $2v>\chi _{3}^{2},$ then by solving (\ref{diff}) and then imposing
initial conditions we get 
\begin{equation*}
\psi (x)=\frac{2\sin r\alpha }{2\alpha \cos r\alpha -\chi _{3}\sin r\alpha },
\end{equation*}%
where we denoted $\alpha \allowbreak =\allowbreak \frac{1}{2}\sqrt{2v-\chi
_{3}^{2}}\allowbreak =\allowbreak \frac{1}{2}\sqrt{2\chi _{4}-3\chi _{3}^{2}}%
.$ Thus $f(x)\allowbreak =\allowbreak x(c_{1}-\frac{\chi _{3}}{2\alpha
^{2}+\chi _{3}^{2}/2})\allowbreak \allowbreak +\allowbreak \frac{1}{v}(2%
\func{arctanh}(\frac{\chi _{3}}{2\alpha }\tan (x\alpha ))-\ln (2\alpha
^{2}-\chi _{3}^{2}/2+(2\alpha ^{2}+\chi _{3}^{2}/2)\cos 2x\alpha ).$ Now
recall that $\func{arctanh}x\allowbreak =\allowbreak \frac{1}{2}\ln \frac{1+x%
}{1-x}$ and we get (\ref{og1}) and (\ref{og2})

\item If $2v<\chi _{3}^{2},$ then we argue in the same way as in the
previous case but in this case parameter $\alpha $ is imaginary and we get
hyperbolic functions.
\end{enumerate}
\end{proof}

\end{document}